\documentclass[10pt,reqno]{amsart}

\usepackage[margin=10pt,font=small,labelfont=bf,justification=centering]{caption}
\usepackage{lipsum}
\usepackage{caption}
\usepackage{subcaption}

\usepackage{amscd}
\usepackage{amsthm}
\usepackage{amssymb}

\usepackage{enumerate,array,mathrsfs, a4wide}
\usepackage{graphicx}
\usepackage{sseq}
\usepackage{xcolor}

\usepackage{rotating}
\usepackage{bbm}
\usepackage[all,cmtip]{xy}

\usepackage[latin1]{inputenc}
\usepackage{dsfont}

\usepackage{hyperref}

\newtheorem{thm}{Theorem}[section]

\newtheorem{lemma}[thm]{Lemma}
\newtheorem{prop}[thm]{Proposition}
\newtheorem{cor}[thm]{Corollary}

\newtheorem*{theorem*}{Theorem}

\theoremstyle{definition}
\newtheorem{fed}[thm]{Definition}

\theoremstyle{remark}
\newtheorem{rem}[thm]{Remark}

\numberwithin{equation}{section}

\newcommand{\C}{{\mathbb{C}}}
\newcommand{\R}{{\mathbb{R}}}

\renewcommand{\P}{{\mathbb{P}}}
\newcommand{\Z}{{\mathbb{Z}}}

\renewcommand{\epsilon}{\varepsilon}
\renewcommand{\phi}{\varphi}
\renewcommand{\theta}{\vartheta}

\DeclareMathOperator{\Aut}{Aut}

\DeclareMathOperator{\SO}{SO}

\begin{document}
\title{Symmetric periodic orbits and uniruled real Liouville domains}
\author[Urs Frauenfelder]{Urs Frauenfelder}

\address{
Department of Mathematics and
Research Institute of Mathematics, Seoul National University\\
Building 27, room 403\\
San 56-1, Sillim-dong, Gwanak-gu, Seoul, South Korea\\
Postal code 151-747
}
\email{frauenf@snu.ac.kr}

\author[Otto van Koert]{Otto van Koert}

\address{
Department of Mathematics and
Research Institute of Mathematics, Seoul National University\\
Building 27, room 402\\
San 56-1, Sillim-dong, Gwanak-gu, Seoul, South Korea\\
Postal code 151-747
}
\email{okoert@snu.ac.kr}



\begin{abstract}
A real Liouville domain is a Liouville domain together with an exact anti-symplectic involution. We call
a real Liouville domain uniruled if there exists an invariant finite energy plane through every real point.
Asymptotically an invariant finite energy plane converges to a symmetric periodic orbit. In this note
we work out a criterion which guarantees uniruledness for real Liouville domains.
\end{abstract}

\maketitle

\section{Introduction}

A {\bf real Liouville domain} $(W,\lambda,\varrho)$ is a triple consisting of a Liouville
domain $(W,\lambda)$ and an exact anti-symplectic involution $\varrho \in \mathrm{Diff}(W)$, i.e.~a map $\varrho$ satisfying
$$
\varrho^2=\mathrm{id}, \quad \varrho^* \lambda=-\lambda.
$$
If we restrict $\varrho$ to the boundary $\partial W$ of the Liouville domain $W$ we get a
real contact manifold, meaning a contact manifold together with an involution under which the contact form is anti-invariant. 
If $R$ denotes the Reeb vector field on $\partial W$, then $R$ is anti-invariant under $\varrho$ as well, i.e.
$$
\varrho^* R=-R.
$$
If $T>0$ and $v \in C^\infty([0,T],\partial W)$ is a $T$-periodic orbit for $R$, 
then $v_\varrho \in C^\infty([0,T],\partial W)$ defined as
$$v_\varrho(t)=\varrho(v(T-t))$$
is a $T$-periodic orbit as well.
\begin{fed}
A $T$-periodic orbit $v \in C^\infty([0,T],\partial W)$ is called {\bf symmetric} if 
it satisfies $v=v_\varrho$.
\end{fed}
Symmetric periodic orbits play a prominent role in the restricted three body problem \cite{birkhoff} as well
as in the Seifert conjecture on brake orbits \cite{seifert}. 

The Weinstein conjecture asserts that on every closed contact manifold the Reeb flow admits a periodic orbit. Affirmative answers to this conjecture can be obtained in various cases by taking advantage of
the interplay between holomorphic curves and closed Reeb orbits \cite{hofer-viterbo, liu-tian, lu, wendl}.
To examine this connection in the real case we introduce the notion of a uniruled real Liouville
domain. Note that for a real Liouville domain $(W,\lambda,\varrho)$ the Liouville vector field $X$ defined by the equation $\iota_X d \lambda=\lambda$
is invariant under $\varrho$ and therefore $\varrho$ extends to the completion $V$ of $W$. By abuse
of notation we will use the symbols $\lambda$ and $\varrho$ also for the extensions to $V$. If we
choose on $V$ an SFT-like almost complex structure anti-invariant under $\varrho$, then $\varrho$
induces an involution of finite energy planes on $V$. Inspired by the paper of McLean \cite{mclean}
we make the following definition.
\begin{fed}
\label{realuni}
A real Liouville domain $(W,\lambda, \varrho)$ is called {\bf (real) uniruled} if for
every anti-invariant SFT-like complex structure $J$ on the completion $(V,\lambda,\varrho)$ there
exists an invariant finite energy plane of SFT-energy less than or equal to 1 through every point on the Lagrangian
submanifold $\mathrm{Fix}(\varrho) \subset V$.
\end{fed}

The asymptotic behavior of finite energy planes as studied in \cite{hofer-wysocki-zehnder1, hofer-wysocki-zehnder2, hofer-wysocki-zehnder3, mora} immediately implies
\begin{thm}\label{symper}
Assume that $(W,\lambda,\varrho)$ is a uniruled real Liouville domain. Then there exists a symmetric
periodic orbit of the Reeb vector field $R$ on $\partial W$ of period less than or equal to 1. 
\end{thm}

\begin{rem}
If one requires that the SFT-energy of the invariant finite energy planes in Definition~\ref{realuni} is less than
or equal to a constant $\kappa>0$ instead of being less than or equal to 1, the period of the symmetric
Reeb orbit in Theorem~\ref{symper} can be estimated from above by the constant $\kappa$. However,
we can always scale $\lambda$ to $\frac{1}{\kappa}\lambda$ so that one does not gain anything
by considering this more general notion.
\end{rem}
The purpose of this note is to provide a condition which guarantees uniruledness for a real Liouville domain.
For this we embed the real Liouville domain into a closed symplectic manifold and use Gromov-Witten theory on this ambient manifold. 
One could use Welschinger's invariants (``real Gromov-Witten theory'') as used for instance in \cite{Welschinger}, but we will argue indirectly.
Let us now explain the properties we require on the ambient manifold.

Assume that $(M,\omega)$ is a closed symplectic manifold that satisfies the Bohr-Sommerfeld condition, namely the cohomology class represented by the symplectic form is integral in the sense that the class $[\omega]$ lies in the image of $H^2(M;\Z)$ in $H^2(M;\R)$.
We suppose in addition that $[\omega]$ is {\bf primitive} in the sense that for every $k>1$ the cohomology class $\frac{1}{k}[\omega]$ is not integral. 
\begin{fed} 
We say that a symplectic hypersurface $\Sigma \subset M$ is 
{\bf primitive} if $[\Sigma]$ is Poincar\'e dual to $[\omega]$. 
\end{fed}
\begin{rem}
If $H^2(M;\Z)$ is torsion free, this notion is unambiguous.
If $H^2(M;\Z)$ has torsion, the class $[\omega]\in H^2_{dR}(M)$ does not uniquely determine an integral cohomology class.
In this latter case, we mean that $[\Sigma]$ is Poincar\'e dual to $[\omega]$ when regarded as a real homology class in $H_{2n-2}(M;\R)$.
\end{rem}

Denote by $h \colon \pi_2(M) \to H_2(M;\mathbb{Z})$ the Hurewicz homomorphism.
\begin{fed} We say that a class $A \in \mathrm{im}(h)$
is {\bf decomposable} if there exist classes $B, C \in \mathrm{im}(h)$ satisfying
$$
A=B+C, \quad \langle [\omega],B \rangle>0, \quad \langle [\omega], C \rangle >0.
$$
We say that $A$ is {\bf indecomposable} if it is not decomposable.
\end{fed}

\begin{fed}
A {\bf decoration} 
$\mathcal{D}=(\Sigma,A,S)$
of $(M,\omega)$ is a triple consisting of a primitive symplectic hypersurface $\Sigma \subset M$,
an indecomposable homology class
$A \in H_*(M)$ and a submanifold $S \subset \Sigma$ satisfying the following two requirements
\begin{description}
 \item[(i)] $A \circ [\Sigma]=1$.
 \item[(ii)] The Gromov-Witten invariant $GW_A([S],[p])$ is odd, where $[p]$
is the homology class of a point. 
\end{description}
We refer to the triple $(M,\omega,\mathcal{D})$ as a {\bf decorated symplectic manifold}. 
\end{fed}

\begin{rem}
By Gromov-Witten invariants we mean the variant defined in~\cite{McDuff_Salamon:holomorphic}, and for this we insist that $S$ is a submanifold rather than a general cycle. 
\end{rem}

\begin{rem}
Note that for a decoration $\mathcal{D}=(\Sigma,A,S)$ we have
$$\langle [\omega],A \rangle=PD([\omega]) \circ A=[\Sigma] \circ A=1$$
so that each holomorphic sphere contributing to the Gromov-Witten invariant 
$GW_A([S],[p])$ has symplectic area equal to 1. 
\end{rem}
\begin{fed}
Assume that $(M,\omega,\mathcal{D})$ is a decorated symplectic manifold with decoration
$\mathcal{D}=(\Sigma,A,S)$. An {\bf anti-decorating involution} $\rho \colon M \to M$
is an anti-symplectic involution satisfying the following conditions
\begin{description}
 \item[(i)] Both $\Sigma$ and $S$ are invariant under $\rho$.
 \item[(ii)] $\rho^*A=-A$.
\end{description}
A {\bf decorated real symplectic manifold} $(M,\omega,\mathcal{D},\rho)$ is a quadruple consisting of  a decorated
symplectic manifold $(M,\omega,\mathcal{D})$ together with an anti-decorating involution $\rho$.
\end{fed}

\begin{fed}
Assume that $(W,\lambda,\varrho)$ is a real Liouville domain and $(M,\omega,\mathcal{D},\rho)$
is a decorated real symplectic manifold. An {\bf embedding of a real Liouville domain into a decorated symplectic manifold} 
$$\varepsilon \colon (W,\lambda,\varrho) \to (M,\omega,\mathcal{D}, \rho)$$
is an embedding $\varepsilon \colon W \to M \setminus \Sigma$ satisfying
$$
d\lambda=\varepsilon^* \omega,\qquad \varrho=\varepsilon^* \rho.
$$
A {\bf Christmas tree} is a quadruple $(W,\lambda,\rho,\varepsilon)$
consisting of a real Liouville domain $(W,\lambda,\rho)$ and an embedding
$\varepsilon \colon (W,\lambda,\rho) \to (M,\omega,\mathcal{D},\rho)$ into a decorated
real symplectic manifold. 
\end{fed}
The main result of this paper is
\begin{thm}
\label{main}
Assume that $(W,\lambda,\rho,\varepsilon)$ is a Christmas tree satisfying $b_1(W)=0$.
Then $(W,\lambda,\rho)$ is real uniruled. 
\end{thm}
Combining Theorem~\ref{main} with Theorem~\ref{symper} we obtain the following Corollary.
\begin{cor}
Assume that $(W,\lambda,\rho,\varepsilon)$ is a Christmas tree satisfying $b_1(W)=0$. Then
there exists a symmetric periodic orbit of period less than or equal to 1 for the Reeb flow on
$\partial W$.
\end{cor}

\section{Definitions and notions of symplectic field theory (SFT)}
By a {\bf real symplectic manifold} we mean a triple $(M,\omega,\rho)$ where $(M,\omega)$ is a symplectic manifold
and $\rho \in \mathrm{Diff}(M)$ is an {\bf anti-symplectic involution}, so
$$
\rho^2=\mathrm{id}, \quad \rho^*\omega=-\omega.
$$
A {\bf Liouville domain} is a compact exact symplectic manifold $(W,\omega=d\lambda)$ with a global Liouville vector field, defined by $i_X\omega=\lambda$, such that the boundary is smooth and convex, meaning that the Liouville vector field $X$ points outward at the boundary.

The boundary of a Liouville domain carries a natural cooriented contact structure. 
Indeed, the Liouville condition implies that $\alpha:=\lambda|_{\partial W}$ is a positive contact form on $\partial W$, so $\alpha \wedge (d\alpha)^{n-1}>0$. 
The hyperplane distribution defined by
$$
\xi=\ker \alpha \subset T \partial W
$$
is called the {\bf contact structure} and the vector field $R$ on $\partial W$ defined by the equations
$$
\iota_R \alpha=1, \quad \iota_R d\alpha=0
$$
is called the {\bf Reeb vector field}.

The following procedure can be used to complete a Liouville domain $W$ into a so-called {\bf Liouville manifold}, which has cylindrical ends instead of convex boundary components.
For each boundary component $C$ of $\partial W$, we attach the positive end of a {\bf symplectization}, given by the symplectic manifold $([0,\infty[\times C,d(e^t \alpha)\, )$, to $W$ along $C$.
The Liouville vector field on the cylindrical end is
$$
X=\frac{\partial}{\partial t}
$$
After this process we obtain a complete Liouville manifold, which we will denote by $(V,\lambda)$

An almost complex structure $J$ on a complete Liouville manifold $V$ is called {\bf compatible} with the symplectic form $\omega=d\lambda$ if
$\omega(\cdot, J \cdot)$ is a Riemannian metric. 
An $\omega$-compatible almost
complex structure $J$ is called {\bf SFT-like} if it satisfies the following conditions
\begin{enumerate}
\item $J$ preserves the hyperplane distribution $\xi$ on $\partial W\subset V$.
\item On $\partial W$ it rotates the Liouville vector field into the Reeb vector field in the sense that
$JX=R$ and $JR=-X$.
\item On the cylindrical end $\partial W \times [0,\infty[$ the almost complex structure is
invariant under the Liouville flow $\phi^t_X$ for $t \in [0,\infty)$.
\end{enumerate}
Pick an SFT-like almost complex structure $J$ on $V$ and assume that
$w \colon (\mathbb{C},i) \to (V,J)$ is a $J$-holomorphic plane. 
We now explain how to define the energy of $w$. This will be a variation of the Hofer energy. 
Choose a small $\delta>0$, indicating the size of a collar neighborhood of $\partial W$, and define
$$
\Lambda:=\Big\{\phi \in C^\infty\big(]-\delta,\infty[, [0,1]\big): \phi'\geq 0,\,\,
\phi'|_{]-\delta,0]}=0\Big\}.
$$
For $\phi \in \Lambda$ define a $1$-form $\lambda_\phi \in \Omega^1(V)$ by
\[
\lambda_\phi(y)=
\begin{cases}
\phi(r)\alpha(x) & \text{if }y=(x,r) \in \partial W \times [0,\infty[\\
\phi(0) \lambda(y) & \text{if } y \in W
\end{cases}
\]
and abbreviate $\omega_\phi=d\lambda_\phi$. 
The {\bf Hofer energy} or {\bf SFT energy} of $w$ is then defined as
$$E(w)=\sup_{\phi \in \Lambda} \int_{\mathbb{C}} w^* \omega_\phi \in [0,\infty].$$
The holomorphic plane $w$ is called {\bf a finite energy plane} if it satisfies
$$
0<E(w)<\infty.
$$
We also have the following non-real version of uniruledness, somewhat different from \cite{mclean}.
\begin{fed}
We call a Liouville domain $(W,\lambda)$ {\bf uniruled} if for
every SFT-like almost complex structure $J$ on its completion $(V,\lambda)$ there exists a finite energy plane through every point of $V$.
\end{fed}

\section{Examples of Christmas trees}
In this section we will discuss some examples of Christmas trees.
An interesting example concerns the canonical contact form and structure on the unit cotangent bundle of a sphere, $(T^*S^n,\lambda_{can},\rho)$, which can be embedded as a real Liouville manifold into the projective quadric with various anti-symplectic involutions $\rho$.
We will check that the projective quadric can be decorated by computing a suitable Gromov-Witten invariant.
Real Liouville structures on $T^*S^2$ include the regularized, planar circular restricted three body problem \cite{Albers_Frauenfelder_Koert_Paternain_Liouville_field_for_PCR2BP}, which has one anti-symplectic involution, and the Hill's lunar problem, which has two commuting anti-symplectic involutions.

Before we verify the decoration requirements for the quadric, we start by giving the following basic lemma.
\begin{lemma}
\label{lemma:real_Liouville_complement}
Let $(M,\omega,\mathcal D=(\Sigma,A,B)\,)$ be a decorated symplectic manifold with anti-decorating involution $\rho$.
Then $M-\nu_M(\Sigma)$ carries the structure of a real Liouville domain, where $\nu_M(\Sigma)$ denotes a tubular neighborhood of $\Sigma$ in $M$. 
\end{lemma}

\begin{proof}
We first show that $W:=M-\nu_M(\Sigma)$ is an exact symplectic manifold.
For this, consider the long exact sequence of the pair in cohomology,
$$
H^2(M,W) \stackrel{j_\Sigma^*}{\longrightarrow} H^2(M) \stackrel{j_W^*}{\longrightarrow} H^2(W).
$$
By Corollary 11.2 of \cite{MilnorStasheff}, the cohomology ring $H^*(M,W)$ is canonically isomorphic to the cohomology ring $H^*(\nu_M(\Sigma),\nu_M(\Sigma)_0)$, associated with the normal bundle of $\Sigma$.
Here $\nu_M(\Sigma)_0$ denotes the normal bundle of $\Sigma$ with its zero-section removed.
Thus the Thom class $u \in H^2(\nu_M(\Sigma),\nu_M(\Sigma)_0)$ corresponds to a class $u'$ in $H^2(M,W)$.
As the homology class $[\Sigma]$ is Poincar\'e dual to $[\omega]$ (over the reals), it follows that $j_\Sigma^* u'$ equals $[\omega]$ by Problem 11-C from \cite{MilnorStasheff}.
By exactness of the long exact sequence of the pair, we see ${j_W^*}[\omega]=j^*_W \circ j^*_\Sigma u'=0$, so there exists a $1$-form $\lambda\in \Omega^1_W$ such that $d\lambda=\Omega:=\omega|_{W}$.

We now show that we can choose a real Liouville form $\tilde \lambda$, i.e.~$\rho^* \tilde \lambda=-\tilde \lambda$. 
Since $d\lambda=\Omega$, and $\rho^* \Omega=-\Omega$ we see that there exists a closed $1$-form $\mu$ such that
$$
\rho^* \lambda=-\lambda+\mu.
$$
Since $\lambda=\rho^* \circ \rho^* \lambda=\lambda-\mu+\rho^*\mu$, we see that $\mu=\rho^*\mu$.
Define $\tilde \lambda:=\lambda-\frac{1}{2} \mu$.
Then $\rho^*\tilde \lambda=\rho^*\lambda-\frac{1}{2}\rho^*\mu=-\lambda+\frac{1}{2}\mu=-\tilde \lambda$.
Hence $(W,\tilde \lambda,\rho)$ is the desired real Liouville domain.
\end{proof}

\subsection{Smooth quadrics in projective space}
We define a {\bf quadric} in projective space as the zeroset of a non-zero homogeneous quadratic polynomial.
Note that a homogeneous quadratic polynomial can always be written as $p(z)=z^t B z$, where $B$ is a symmetric matrix.
By Sylvester's theorem, we can assume that $B$ is diagonal.
We then easily see
\begin{lemma}
A quadric is smooth if and only if $B$ has maximal rank.
\end{lemma}

We have the following identification of the smooth projective quadric with an oriented Grassmannian.
\begin{lemma}
\label{lemma:quadric_symmetric_sp}
The smooth projective quadric given by
$$
Q^n=\{ [z_0,\ldots,z_{n+1}]\in \C \P^{n+1}~|~\sum_j z_j^2=0 \}
$$
is diffeomorphic to the symmetric space $Gr^+(2,n)\cong \SO(n+2)/\SO(2)\times \SO(n)$.
Furthermore, $\SO(n+2)$ acts transitively via biholomorphisms.
\end{lemma}

\begin{proof}
For the first part, we exhibit the diffeomorphism
\[
\begin{split}
Gr^+(2,n+2) & \longrightarrow Q^n \\
span(x,y) & \longmapsto x+iy.
\end{split}
\]
Here $x,y\in \R^{n+2}$ form an orthonormal basis of the $2$-plane they span. We use that $\sum_j z_j^2=\Vert x \Vert^2-\Vert y \Vert^2+2i\langle x,y \rangle$.
To see that $\SO(n+2)$ acts by biholomorphisms, just observe that
\[
\begin{split}
\SO(n+2) \times Q^n \longrightarrow Q^n\\
(A,[x+iy]) & \longmapsto [Ax+iAy]=[A(x+iy)].
\end{split}
\]
\end{proof}

By an {\bf affine quadric} we mean the zeroset of a non-zero quadratic polynomial in $\C^{n+1}$.
Away from possible singular points an affine quadric inherits a symplectic structure as a complex submanifold of a K\"ahler manifold.
It is well-known, see \cite[Exercise 6.20]{McDuff_Salamon:introduction}, that a smooth affine quadric is symplectomorphic to $T^*S^n$ with its canonical symplectic structure.
\begin{lemma}
There is a symplectomorphism
\[
\begin{split}
(V=\{ (z_0,\ldots,z_n)\in \C^{n+1}~|~\sum_j {z_j}^2=1 \},\omega_0) & \longrightarrow (T^*S^n,\omega_{can})\subset T^*\R^{n+1} \\
z=x+iy & \longmapsto (\frac{x}{\Vert x\Vert },\Vert x\Vert y).
\end{split}
\]
\end{lemma}
The singular affine quadric appearing in the following lemma is also of interest.
\begin{lemma}
The symplectization of $(ST^*S^n,\lambda_{can})$ is symplectomorphic to 
$$
V_0=\{ (z_0,\ldots,z_n)\in \C^{n+1}~|~\sum_j {z_j}^2=0 \} \setminus \{ 0 \}
.
$$
In addition, the standard complex structure $i$ is an SFT-like complex structure for the symplectization.
\end{lemma}

\subsection{Naive Gromov-Witten invariants of quadrics}
We consider a smooth quadric $Q^n$ given as the zero locus of the symmetric bilinear form $B$.
The Lefschetz hyperplane theorem implies that for $n>2$, we have $H_2(Q^n;\Z)\cong \Z$, see \cite[Example 4.27]{McDuff_Salamon:introduction}.
Moreover, this homology group is generated by a line $L$, by which we mean a map of the form $[\lambda:\mu]\in \C \P^1 \mapsto \lambda p+\mu q$, where $p,q\in Q^n\subset \C \P^{n+1}$ (so $B(p,p)=B(q,q)=0$) and $B(p,q)=0$.
The quadric $Q^2$ in $4$-dimensions is diffeomorphic to $S^2 \times S^2$, so $H_2(Q^2;\Z)\cong \Z^2$, and there are two types of lines, distinguished by their homology class.
We will equip $Q^n$ with its natural complex structure $J_0$.

Let $Hol(J_0,[L])$ denote the space of $J_0$-holomorphic maps from $\C \P^1$ to $Q^n$ representing the homology class $[L]$.
Write ${\mathcal M}(J_0,[L])$ for the moduli space of $J_0$-holomorphic curves with homology class $[L]$.
We have
$$
{\mathcal M}(J_0,[L])=Hol(J_0,[L])/Aut(\C \P^1).
$$
We will compute some Gromov-Witten invariants by ``naive counting'', \cite{RT:counting}.
To show that this works, we needs to establish regularity of $J_0$.

\subsection{Moduli space and regularity}
Let $L$ be a line on a smooth projective quadric with primitive homology class $[L]\in H_2(Q^n;\Z)$.
We linearize the Cauchy-Riemann equations at a parametrization of $L$ given by $u:\C \P^1\to Q^n$.
\begin{lemma}
The linearized operator at $u$ is surjective.
In particular, the space of holomorphic maps $Hol(J_0,[L])$ in $Q^n$ is a smooth manifold of dimension $\dim Hol(J_0,[L])=2n+2n$. 
\end{lemma}
We give two arguments for this statement.

\subsubsection{Regularity via sheafs and splitting of the normal bundle}
In the language of sheafs, triviality of the cokernel is equivalent to vanishing of the sheaf cohomology group $H^1(L,\mathcal T Q^n|_L)$ (cf.~the statement of Riemann-Roch).
We have the short exact sequence of sheafs
$$
0 \longrightarrow \mathcal T L \longrightarrow \mathcal T Q^n|_L \longrightarrow \nu_L \longrightarrow 0,
$$
where $\nu_L$ is the sheaf of germs of holomorphic sections of the normal bundle of $L$.
A piece of the corresponding long exact sequence in cohomology looks like
$$
H^1(L,\mathcal T L) \longrightarrow H^1(L,\mathcal T Q^n|_L) \longrightarrow H^1(L,\nu_L).
$$
It is a well-known classical fact that $H^1(\C\P^1,\mathcal O(k)\,)=0$ for $k\geq -1$ (a generalization of this formula is known as the Bott formula, see \cite[Chapter 1]{OSS:cpx_vb}), so we see directly that $H^1(L,\mathcal T L )=0$ as $\mathcal T L\cong \mathcal O(2)$.
For the normal bundle, note that a line $L$ in a smooth quadric $Q^n$ is always contained in a tower of smooth quadrics of the form
$$
L\subset Q^2 \subset Q^3 \subset \ldots \subset Q^n.
$$
The normal bundle $\nu_{Q^k}(Q^{k-1})$ is isomorphic to $\mathcal O(1)$, and the normal bundle $\nu_{Q^2}(L)$ is trivial, so $\nu_L$ splits as
$$
\mathcal O(1)^{n-2}\oplus \mathcal O.
$$
By the earlier mentioned Bott formula $H^1(L,\nu_L)\cong H^1(\C\P^1,\mathcal O(1)\,)^{\oplus n-2}\oplus H^1(\C\P^1,\mathcal O\,)=0$, so we conclude that $H^1(L,\mathcal T Q^n|_L)=0$.

\subsubsection{Regularity via holomorphic transitive actions}
Lemma~\ref{lemma:quadric_symmetric_sp} tells us that we have a holomorphic transitive action on $Q^n$, so by \cite[Proposition 7.4.3]{McDuff_Salamon:holomorphic}, every holomorphic sphere is regular, and the claim of the Lemma follows.

\subsection{Lines through a point}
Now consider the evaluation map
\[
\begin{split}
ev:Hol(J_0,[L]) \times_{\Aut(\C \P^1)} \C \P^1 & \longrightarrow Q^n \\
[u,z] & \longmapsto u(z).
\end{split}
\]
By Sard's theorem we find a regular value $p$ of $ev$, and in fact, since $SO(n+2)$ acts transitively on $Q^n$, every value is regular.
Define the moduli space of lines through $p$ as $\mathcal M_p=ev^{-1}(p)$.

Geometrically, we can describe $\mathcal M_p$ as follows.
If $L=pq$ is a line through $p$ and $q$ that is completely contained in $Q^n$, then $B(\lambda p+\mu q,\lambda p+\mu q)=0$ for all $[\lambda:\mu]\in \C \P^1$.
This gives a quadratic equation in $\lambda$ and $\mu$, which should vanish identically, so by looking at the coefficients we find
$$
B(p,p)=0,\quad B(p,q)=0,\quad B(q,q)=0.
$$
As $p$ and $q$ lie on $Q^n$, we automatically have $B(p,p)=0=B(q,q)$.
The remaining equation defines a hyperplane in $\C \P^{n+1}$, namely the ``geometric tangent plane''
$$
P:=\{ z\in \C \P^{n+1}~|~B(p,z)=0 \}
.
$$
Since every line through $p$ intersects the quadric at infinity, given by $Q_\infty=\{ z=[z_0:\ldots:z_{n}:0]~|~z\in Q^n \}$, we can identify the moduli space of lines through $p$ with $\mathcal M_p=Q_\infty \cap P$.

To obtain a Gromov-Witten invariant, we will consider lines through $p$ going through an additional cycle $C$.
First define
\[
\begin{split}
ev:Hol(J_0,[L]) \times_{\Aut(\C \P^1)} \C \P^1 \times \C \P^1 & \longrightarrow Q^n\times Q^n \\
[u;z_1,z_2] & \longmapsto (u(z_1),u(z_2)\,).
\end{split}
\]
A dimension count tells us that $C$ should be a $2$-cycle if we want $ev^{-1}(\{ p \} \times C )$ to consist of points.
Hence we take $C$ to be a line (which is of course a smooth submanifold) in $Q_\infty$ which transversely intersects ${\mathcal M}_p$, regarded as a subset in $Q_\infty$, in a point $q_0$.
We get a unique element in ${\mathcal M}(J_0,[L]) \times_{\Aut(\C \P^1)} \C \P^1 \times \C \P^1$ which maps to $(p,q)\in Q^n \times Q^n$, and we may represent this element by $(u;[0:1],[1:0])$.

To check that the evaluation map is transverse to $\{ p \} \times C$, we observe that $C$ is transverse to the set
$$
Cone(p, \mathcal M_p)=\{ q\in Q^n~|~q \text{ lies on the line from }p \text{ to some point in }\mathcal M_p\subset Q_\infty \subset Q^n \}
$$
First we show that vectors of the form $(v,0)\in T_p Q^n \times T_{q_0} Q^n$ lie in the image of $T_{[u;[0:1],[1:0]\,]}ev$.
Indeed, put $p_s:=\exp_p(s v)$, and follow the above procedure to define $\mathcal M_{p_s}$.
For small $s$ we find a unique intersection point $q_s:=\mathcal M_{p_s}\cap C$.
Therefore we find a variation $(u_s,[0:1],[1:0])$ which is mapped to $(p_s,q_s)$ under $ev$. Note here that the curve $q_s$ is tangent to $C$. 

To see that a vector of the form $(0,w)$ also lies in the image of $T_{[u;[0:1],[1:0]\,]}ev$, we first note that we can assume that $w$ lies in the tangent space to $Cone(p, \mathcal M_p)$ since the normal to $Cone(p, \mathcal M_p)$ is tangent to $C$.
The curve $\tilde q_s:=\exp_p(s w)$ lies in $Cone(p, \mathcal M_p)$, so by definition of this cone, we find a line from $p$ to $\tilde q_s$.
Hence we find a variation $(u_s,[0:1],[1:z_s])$ which maps to $(p,\tilde q_s)$.

We conclude
\begin{prop}
The $2$-point Gromov-Witten invariant $GW^{Q^n}_{[L]}([p],[C])$ equals $1$.
\end{prop}
We remind the reader that $H_2(Q^2)\cong \Z\oplus \Z$, and there are two distinct homology classes $[L]$ represented by a line in this case.
We collect the above results in the following theorem.
\begin{thm}
The projective quadric $Q^n$ admits a decoration by $\mathcal D=(Q^{n-1},[L],C)$, where $[L]$ is the homology class of a line and $C$ is the submanifold described above.
\end{thm}

\begin{rem}
It is clear that the projective quadric has many anti-symplectic involutions.
For instance, we can compose conjugation with swapping coordinates.
\end{rem}

\section{Existence of invariant curves}
Complex conjugation on $\C \P^1$ defines an anti-symplectic involution $R_0 \colon \mathbb{CP}^1 \to \mathbb{CP}^1$, namely
$$
\rho_0[z_0:z_1]=[\bar{z}_0:\bar{z}_1].
$$
Now pick an $\omega$-compatible almost complex structure $J$ on $TM$ which is anti-invariant under $\rho$, so
$$\rho^* J=-J.$$
Denote the space of parametrized $J$-holomorphic maps from $\mathbb{CP}^1$ to $M$  by $Hol(J)$.
We define an involution on this space,
\[
\begin{split}
I\colon Hol(J) & \longrightarrow Hol(J) \\
u & \longmapsto \rho \circ u \circ \rho_0.
\end{split}
\]
Now we will write the fixed point locus of this involution as
$$
Hol(J)^\rho=\{ u \in Hol(J) : I(u)=u\}
.
$$
Take a point $p \in M$, a submanifold $S \subset M$ 
and a spherical homology class $A \in \mathrm{im}(h)$, where $h \colon \pi_2(M) \to H_2(M;\mathbb{Z})$ is the Hurewicz homomorphism, and define
$$
Hol(J;(S,p;A)\,)=\big\{ u \in Hol(J): u(\nu) \in S, u(\sigma)=p, [u]=A\big\}
$$
where $\nu=[1:0] \in \mathbb{CP}^1$ is the ``north-pole'' and $\sigma=[0:1] \in \mathbb{CP}^1$
is the ``south-pole''. 
Note that both the north- and the south-pole lie on the real part $\mathbb{RP}^1=\mathrm{Fix}(\rho_0)
\subset \mathbb{CP}^1$. 
The parametrization is not yet fully determined by just two marked points, so we still have a $\C^*$-action on this space.
Later, we will mod out by this action.

Suppose now that $S$ is invariant under $\rho$, that the point $p$ lies in the Lagrangian $L=\mathrm{Fix}(\rho)$, and that the homology class $A$ is anti-invariant, so $\rho_* A=-A$.
Then the space $Hol(J,(S,p;A)\,)$ is invariant under the involution $I$ and we set
$$
Hol^\rho(J,(S,p;A)\,)=Hol(J,(S,p;A)\,) \cap Hol(J)^\rho.
$$
If $\Sigma \subset M$ is a symplectic submanifold we will write $\mathscr{J}(\Sigma,\rho)$ for the space of all $\omega$-compatible almost complex structures on $M$, which are anti-invariant under the anti-symplectic involution $\rho$ and which restrict on $\Sigma$ to an $\omega|_{\Sigma}$-compatible almost complex structure such that $\Sigma$ becomes a $J$-holomorphic submanifold of $M$. The main result of this section is
the following theorem.
\begin{thm}\label{realmain}
Assume that $(M,\omega,\mathcal{D},\rho)$ is a decorated real symplectic manifold
with decoration $\mathcal{D}=(\Sigma,A,S)$.
Then for every point $p \in L \cap \Sigma^c$ and every almost complex structure
$J \in \mathscr{J}(\Sigma,\rho)$ the moduli space $\mathcal{M}_J^\rho(S,p;A)=Hol^\rho(J,(S,p;A)\,)/\C^*$ is nonempty.
\end{thm}
The proof of Theorem~\ref{realmain} needs some preparation. 
We first
recall from \cite[Section\,2.5]{McDuff_Salamon:holomorphic} that a holomorphic curve
$u \colon \mathbb{CP}^1 \to M$ is called {\bf multiply covered} if there exists
a holomorphic curve $v \colon \mathbb{CP}^1 \to M$ and a holomorphic map $\phi \colon \mathbb{CP}^1 \to \mathbb{CP}^1$ satisfying
$$u=v \circ \phi, \qquad \mathrm{deg}(\phi)>1.$$ 
If a curve is not multiply covered, it is called {\bf simple}. 
\begin{lemma}\label{simpsimp}
A holomorphic curve $u \in Hol(J)$ is simple if and only if $I(u)$ is simple.
\end{lemma}
\begin{proof}
First suppose that $u$ is simple and suppose that $v \in Hol(J)$ and
$\phi \colon \mathbb{CP}^1 \to \mathbb{CP}^1$ is a holomorphic map such that
$$
I(u)=v \circ \phi.
$$
By using that $I$ is an involution, we compute
$$
u=I^2(u)=I(v \phi)=\rho v \phi \rho_0=\rho v\rho_0 \rho_0 \phi \rho_0=I(v) \circ (\rho_0 \phi \rho_0).
$$
Since $u$ is simple by assumption we conclude that
$$
\mathrm{deg}(\phi)=\mathrm{deg}(\rho_0 \phi \rho_0)=1
$$
and therefore $I(u)$ is simple as well. This proves the "only if" part and the "if" part follows
again from the fact that $I^2(u)=u$. 
\end{proof}

We now need that fact that $Aut(\C \P^1)=PSL_2(\C)$.
\begin{fed}
A simple holomorphic curve $u \in Hol(J)$ is called a {\bf pseudo-fixed point} if there exists
$\phi \in PSL_2(\mathbb{C})$ such that $I(u)=u \circ \phi$. It is called a {\bf fixed point}
if $\phi$ is the identity, i.e.~$I(u)=u$.
\end{fed}
\begin{rem}
It follows from \cite[Proposition 2.5.1]{McDuff_Salamon:holomorphic} that a simple holomorphic curve
has no nontrivial automorphisms. Therefore the map $\phi$ for a pseudo-fixed point is uniquely determined.
\end{rem}
\begin{lemma}
Assume that $u \in Hol(J)$ is a pseudo-fixed point, so that $I(u)=u \phi$
for some $\phi \in PSL_2(\mathbb{C})$. Then $\phi \rho_0 \colon \mathbb{CP}^1 \to \mathbb{CP}^1$ is an anti-holomorphic involution.
\end{lemma}
\begin{proof}
That $\phi \rho_0$ is anti-holomorphic is clear. To check that it is an involution we compute
$$u=I^2(u)=I(u  \phi)=\rho u \phi \rho_0=\rho u\rho_0 \rho_0 \phi \rho_0=I(u)\rho_0 \phi \rho_0=u \phi \rho_0 \phi \rho_0.$$
Since $u$ is simple by assumption it follows from \cite[Proposition 2.5.1]{McDuff_Salamon:holomorphic}
that $u$ has no nontrivial automorphisms so that 
$$(\phi \rho_0)^2=\mathrm{id}.$$
This finishes the proof of the lemma.
\end{proof} 
We abbreviate by $\mathcal{I} \subset \mathrm{Diff}(\mathbb{CP}^1)$ the space of
anti-holomorphic involutions of $\mathbb{CP}^1$.
\begin{prop}\label{twocomp}
The space $\mathcal{I}$ has two connected components. 
\end{prop}
\begin{proof}
We first show that $\mathcal{I}$ is diffeomorphic to the space
$$\mathcal{J}=\big\{[A] \in PSL_2(\mathbb{C}): [\bar{A}]=[A^{-1}]\big\}$$
where for $A \in SL_2(\mathbb{C})$ we denote by $[A]$ its equivalence class in the
projectivization $PSL_2(\mathbb{C})$ and by $\bar{A}$ the complex conjugate of the matrix $A$.
We define a map
$$
\Phi \colon \mathcal{I} \to \mathcal{J}, \qquad \psi \mapsto \psi \rho_0.
$$
To check that this map is well defined we first note that $\psi \rho_0 \colon \mathbb{CP}^1 \to \mathbb{CP}^1$ is a biholomorphism so that $\psi \rho_0=[A] \in PSL_2(\mathbb{C})$. 
Now we compute using the fact that $\rho_0$ as well as $\psi$ are involutions
$$
[\bar{A}]=\rho_0(\psi \rho_0) \rho_0=\rho_0 \psi=\rho_0^{-1} \psi^{-1}=(\psi \rho_0)^{-1}=[A^{-1}].
$$
This proves that $\Phi$ is well defined. To show that it is a diffeomorphism we construct its inverse as
follows
$$
\Psi \colon \mathcal{J} \to \mathcal{I}, \qquad \phi \mapsto \phi \rho_0.
$$
That $\Psi$ is inverse to $\Phi$ is an immediate consequence from the fact that $\rho_0$ is an involution.
It therefore just remains to check that $\Psi$ is well defined, i.e.~that $\phi \rho_0$ is actually
an involution. This follows from the following computation
$$(\phi \rho_0)^2=\phi (\rho_0 \phi \rho_0)=\phi \phi^{-1}=\mathrm{id}.$$
This proves that $\mathcal{I}$ and $\mathcal{J}$ are diffeomorphic. 

In view of the diffeomorphism established above we are left with showing that $\mathcal{J}$
has two connected components. We rewrite $\mathcal{J}$ first as the quotient
$$
\mathcal{J}=\widetilde{\mathcal{J}}/\mathbb{Z}_2
$$
where
$$
\widetilde{\mathcal{J}}=\widetilde{\mathcal{J}}_+ \cup \widetilde{\mathcal{J}}_-
$$
with
$$
\widetilde{\mathcal{J}}_\pm=\Bigg\{ A=\left(\begin{array}{cc}
a & b\\
c & d
\end{array}\right) \in SL_2(\mathbb{C}): \left(\begin{array}{cc}
a & b\\
c & d
\end{array}\right)=\pm \left(\begin{array}{cc}
\bar{d} & -\bar{b}\\
-\bar{c} & \bar{a}
\end{array}\right)\Bigg\}
$$
and the $\mathbb{Z}_2$-action identifies $A$ with $-A$. Note that both $\widetilde{\mathcal{J}}_+$ and
$\widetilde{\mathcal{J}}_-$ are invariant under the $\mathbb{Z}_2$-action. If 
$A=\left(\begin{array}{cc}
a & b\\
c & d
\end{array}\right) \in \widetilde{\mathcal{J}}_+$, then this is equivalent that
$$a=\bar{d}, \quad b,c \in i \mathbb{R}, \quad |a|^2-bc=1.$$
Hence we can identify $\widetilde{\mathcal{J}}_+$ with the hyperboloid of one sheet
$$\mathcal{H}_1=\big\{(x_1,x_2,x_3,x_4) \in \mathbb{R}^4: x_1^2+x_2^2+x_3^2-x_4^2=1\big\}$$
via the map
$$\mathcal{H}_1 \to \widetilde{\mathcal{J}}_+, \quad
(x_1,x_2,x_3,x_4) \mapsto\left(\begin{array}{cc}
x_1+i x_2 & i(x_3+x_4)\\
i(x_3-x_4) & x_1-ix_2
\end{array}\right).$$
The hyperboloid of one sheet $\mathcal{H}_1$ is connected and therefore we conclude that
$\widetilde{\mathcal{J}}_+$ and $\widetilde{\mathcal{J}}_+/\mathbb{Z}_2$ are connected as well. 
\\ \\
It remains to show that $\widetilde{\mathcal{J}}_-/\mathbb{Z}_2$ is connected as well. 
 If 
$A=\left(\begin{array}{cc}
a & b\\
c & d
\end{array}\right) \in \widetilde{\mathcal{J}}_-$, then this is equivalent that
$$a=-\bar{d}, \quad b,c \in \mathbb{R}, \quad |a|^2-bc=1.$$
Hence we can identify $\widetilde{\mathcal{J}}_+$ with the hyperboloid of two sheets
$$\mathcal{H}_2=\big\{(x_1,x_2,x_3,x_4) \in \mathbb{R}^4: -x_1^2-x_2^2-x_3^2+x_4^2=1\big\}$$
via the map
$$\mathcal{H}_2 \to \widetilde{\mathcal{J}}_-, \quad
(x_1,x_2,x_3,x_4) \mapsto\left(\begin{array}{cc}
x_1+i x_2 & x_3+x_4\\
x_3-x_4 & -x_1+ix_2
\end{array}\right).$$
The pullback of the involution on $\widetilde{\mathcal{J}}_-$ to $\mathcal{H}_2$ is given by
$x \mapsto -x$. This involution interchanges the two sheets of $\mathcal{H}_2$ and therefore
$\widetilde{\mathcal{J}}_-/\mathbb{Z}_2$ is connected. This finishes the proof of the Proposition.
\end{proof}

Keeping the notation from the proof of Proposition~\ref{twocomp}, we abbreviate the two connected
components of the space $\mathcal{I}$ by
$$\mathcal{I}_\pm:=\Psi(\mathcal{J}_\pm), \quad \mathcal{J}_\pm:=\widetilde{\mathcal{J}}_\pm/\mathbb{Z}_2.$$
An example of a holomorphic involution in $\mathcal{I}_+$ is the involution 
$\rho_0 \colon [z_0:z_1] \mapsto [\bar{z}_0:\bar{z}_1]$ and an example 
of an anti-holomorphic involution in $\mathcal{I}_-$ is the antipodal map
$\sigma_0 \colon [z_0:z_1] \mapsto [\bar{z}_1:-\bar{z}_0]$. Note that the fixed point set of
$\rho_0$ is topologically a circle, while $\sigma_0$ has no fixed points. Since the topological type of the fixed point set only depends on the connected component of $\mathcal{I}$ we conclude the following lemma.
\begin{lemma}\label{fix}
Each anti-holomorphic involution in $\mathcal{I}_-$ acts freely, while the fixed point set of
each involution in $\mathcal{I}_+$ is topologically a circle. 
\end{lemma}
\begin{fed}
A pseudo-fixed point $u \in Hol(J)$ satisfying $I(u)=u \phi$ is called of {\bf type I}
if $\phi \in \mathcal{J}_+$. 
Otherwise $u$ is called of {\bf type II}, meaning that $\phi \in \mathcal{J}_-$.
\end{fed}

\begin{prop}\label{pseufifi}
Assume that $u \in Hol(J)$ is a pseudo-fixed point of type I. Then there exists $\psi \in 
PSL_2(\mathbb{C})$ such that $u \circ \psi$ is a fixed point. 
\end{prop}
\textbf{Proof: } Since $u$ is a pseudo-fixed point we have $I(u)=u \phi$ for 
$\phi \in PSL_2(\mathbb{C})$ and because $u$ is of type I we have $\phi \rho_0 \in \mathcal{I}_+$.
By Lemma~\ref{fix} we know that the fixed point set of $\phi \rho_0$ is topologically a circle. Identify 
$\mathbb{CP}^1$ with the two dimensional sphere $S^2=\{x \in \mathbb{R}^3: ||x||=1\}$ via stereographic projection. We first
claim that the fixed point set $\mathrm{Fix}(\phi \rho_0)$ is actually a small circle, namely the intersection
of $S^2$ with an affine plane in $\mathbb{R}^3$. To see this pick three points on
$\mathrm{Fix}(\phi \rho_0)$. These three points uniquely determine a small circle. Since $\phi$ and $\rho_0$
as well map small circles to small circles we conclude that this small circle is fixed under $\phi \rho_0$
and hence has to agree with $\mathrm{Fix}(\phi \rho_0)$. This shows that $\mathrm{Fix}(\phi \rho_0)$
is a small circle.
\\ \\
Since the group $PSL_2(\mathbb{C})$ acts transitively on small circles we conclude that there exists
$\psi \in PSL_2(\mathbb{C})$ satisfying
$$\psi\big(\mathrm{Fix}(\rho_0)\big)=\mathrm{Fix}\big(\phi \rho_0\big).$$
This implies that
$$\mathrm{Fix}(\phi \rho_0)=\mathrm{Fix}(\psi \rho_0 \psi^{-1}).$$
By analyticity we conclude that
$$\phi \rho_0=\psi \rho_0 \psi^{-1}.$$
Using this equality we compute
$$I(u\psi)=\rho u \psi \rho_0=\rho u \rho_0 \rho_0 \psi \rho_0=u \phi \rho_0 \psi \rho_0=u \psi \rho_0 \psi^{-1} \psi \rho_0=u \psi.$$
Hence $u \psi$ is a fixed point. This finishes the proof of the proposition. \hfill $\square$
\begin{prop}\label{pseuII}
Assume that $\Sigma \subset M$ is a complex $\rho$-invariant hypersurface and $u \in Hol(J)$ is a pseudo-fixed point satisfying $[u] \circ [\Sigma]=1$ and $\mathrm{im}(u) \not\subset \Sigma$. Then $u$ is of type I.  
\end{prop}
\textbf{Proof: } Since $[u] \circ [\Sigma]=1$, the image of $u$ is not contained in $\Sigma$ and $\Sigma$ is complex we deduce from positivity of intersections that
$\#u^{-1}(\Sigma)=1$, i.e.~there exists $w_0 \in \mathbb{CP}^1$ such that
\begin{equation}\label{unique}
u^{-1}(\Sigma)=\{w_0\}.
\end{equation}
Since $u$ is a pseudo-fixed point there exists $\phi \in PSL_2(\mathbb{C})$ such that
$I(u)=u \phi$. 
We compute using the $\rho$-invariance of $\Sigma$
$$u \phi \rho_0(w_0)=\rho u\rho_0 \rho_0(w_0)=\rho u(w_0) \in \rho \Sigma=\Sigma.$$
We deduce from (\ref{unique}) that
$$\phi \rho_0(w_0)=w_0.$$
In particular, the fixed point set of the anti-holomorphic involution $\phi \rho_0$ is not empty. We conclude 
with Lemma~\ref{fix} that $\phi \rho_0 \in \mathcal{I}_+$ or equivalently that $\phi \in \mathcal{J}_+$
and therefore $u$ is a pseudo-fixed point of type I. This proves the proposition. \hfill $\square$

\begin{fed}
Assume $u \in Hol(J)$. 
A point $w \in \mathbb{CP}^1$ is called a {\bf $\rho$-injective
point} of $u$ if
$$
du(w) \neq 0, \qquad u^{-1}\{u(w), \rho u(w))\}=\{w\}.
$$
\end{fed}
\begin{lemma}\label{rinj}
Assume that $u \in Hol(J)$ is a simple holomorphic map which is not a pseudo-fixed point. Then
the complement of the set of $\rho$-injective points of $u$ is finite.
\end{lemma}

\begin{proof}
Denote by $Z_\rho \subset \mathbb{CP}^1$ the complement of the set of $\rho$-injective points. Abbreviate further
$$Z=\big\{w \in \mathbb{CP}^1: du(w)=0\,\,\textrm{or}\,\,\#u^{-1}(u(w))>1\big\}$$
the set of non-injective points of $u$ and
$$\mathcal{T}=\big\{(w_0,w_1) \in \mathbb{CP}^1 \times \mathbb{CP}^1: u(w_0)=\rho u(w_1),\,\,
w_0 \neq w_1\big\}.$$
Consider the map
$$
\pi \colon \mathcal{T} \to \mathbb{CP}^1, \quad \pi(w_0,w_1)=w_0.
$$
Note that
$$
Z_\rho=Z \cup \mathrm{im}(\pi).
$$
Since $u$ is simple the set $Z$ is finite by positivity of intersection, see \cite[Theorem E.1.2.]{McDuff_Salamon:holomorphic}. It therefore suffices to show that the set $\mathcal{T}$ is finite as well.
To see that first note that by Lemma~\ref{simpsimp} $I(u)$ is simple as well. Therefore it
follows from \cite[Corollary 2.5.3]{McDuff_Salamon:holomorphic} that 
$$
\mathrm{im}(u) \neq \mathrm{im}(I(u)).
$$
Hence by positivity of intersection
$$
\#\big\{(w_0,w_1) \in \mathbb{CP}^1 \times \mathbb{CP}^1: u(w_0)=I(u)(w_1)\big\}<\infty.
$$
Note that
\begin{eqnarray*}
& &\#\big\{(w_0,w_1) \in \mathbb{CP}^1 \times \mathbb{CP}^1: u(w_0)=I(u)(w_1)\big\}=\\
& &\#\big\{(w_0,w_1) \in \mathbb{CP}^1 \times \mathbb{CP}^1: u(w_0)=\rho u(w_1)\big\}
\end{eqnarray*}
We deduce that
$$
\#\mathcal{T}<\infty.
$$
This finishes the proof of the Lemma. 
\end{proof}
We are now ready to prove the main result of this section.
\begin{proof}[Proof of Theorem~\ref{realmain}:] 
We argue by contradiction and assume that there
exists $J \in \mathscr{J}(\Sigma, \rho)$ such that the moduli space $\mathcal{M}^\rho_J(S,p;A)=Hol^{\rho}({J'},(S,p;A)\, )/\C^*$ is empty. 
Since
$A$ is indecomposable there is no bubbling and therefore it follows from compactness of
holomorphic curves that there exists an open neighborhood $\mathscr{J}_0 \subset \mathscr{J}(\Sigma,\rho)$ of $J$ such that $\mathcal{M}^\rho_{J'}(S,p;A)=\emptyset$ for every $J' \in \mathscr{J}_0$. 
In view of Proposition~\ref{pseufifi} there is therefore no pseudo-fixed point of type I
in the space of holomorphic maps $Hol^{\rho}(J',(S,p;A)\,)$ for every $J' \in \mathscr{J}_0$.
Together with Proposition~\ref{pseuII} the assumptions of the theorem show that there does not exist a pseudo-fixed point of type II either and 
therefore there are no pseudo-fixed points at all in $Hol({J'},(S,p;A)\,)$ for every $J' \in \mathscr{J}_0$. 

Furthermore, $A$ is indecomposable, so each holomorphic curve $u$ representing $A$ is simple and we conclude with Lemma~\ref{rinj} that for every $J' \in \mathscr{J}_0$
every holomorphic map $u \in Hol({J'},(S,p;A)\,)$ has $\rho$-injective points.
Transversality arguments, see \cite[Section 6.2, Section 6.3]{McDuff_Salamon:holomorphic}, then show
that there exists an open and dense subset $\mathscr{J}_0^{\mathrm{reg}} \subset \mathscr{J}_0$
such that for every $J' \in \mathscr{J}_0^{\mathrm{reg}}$ the Gromov-Witten invariant 
$GW_A([S],[p])$ can be obtained as the signed count of points in the moduli space $\mathcal{M}({J'},(S,p;A)\,)=Hol({J'},(S,p;A)\, )/\C^*$. 
Since this Gromov-Witten invariant is odd by assumption we conclude that
$$
1=GW_A\big([S],[p]\big)\,\mathrm{mod}\,2=\#\mathcal{M}_{J'}(S,p;A)\,\mathrm{mod}\,2.
$$
However, the moduli space $\mathcal{M}_{J'}(S,p;A)$ is invariant under the involution $I$ which has
no fixed points by construction. 
Therefore the cardinality of the moduli space $\mathcal{M}_{J'}(S,p;A)$ has to be even.
This contradiction finishes the proof of the theorem.
\end{proof}

\section{The proof}
The basic idea to prove Theorem~\ref{main} is to embed a real Liouville domain into a decorated symplectic manifold making it into a Christmas tree.
By hanging up some Christmas balls, or in other words taking holomorphic spheres through $\Sigma$ and a given real point $p$, and applying a stretching construction we obtain an invariant finite energy plane through every point in the real locus.

We need some lemmas to prepare the Christmas tree for the Christmas balls.
\begin{lemma}
Let $(M,\omega,\mathcal D=(\Sigma,A,S)\,)$ be a decorated symplectic manifold with an anti-symplectic involution $\rho$, and assume that $(W_0,\lambda_0,\rho|_{W_0})$ is a real Liouville domain that embeds into the interior of $M-\nu(\Sigma)$ for some $\rho$-invariant neighborhood $\nu(\Sigma)$ of $\Sigma$.
Suppose in addition that $b_1(W_0)=0$.

Then $W_1:=M-\nu(\Sigma)$ carries the structure of a real Liouville domain $(W_1,\lambda_1,\rho|_{W_1})$ such that $(W_0,\lambda_0,\rho|_{W_0})$ is a real Liouville subdomain in the sense that $\lambda_1|_{W_0}=\lambda_0$.
\end{lemma}

\begin{proof}
Since we will need a cutoff function, we first extend $\lambda_0$ to a neighborhood of $W_0$.
By Lemma~\ref{lemma:real_Liouville_complement}, $W_1:=M-\nu(\Sigma)$ is a real Liouville domain $(W_1,\tilde \lambda_1,\rho)$.
As $\omega=d\tilde \lambda_1=d \lambda_0$ on a neighborhood of $W_0$, we see that $\tilde \lambda_1-\lambda_0$ is closed, and as $b_1(W)=0$, we find a function $f$ on a neighborhood of $W_0$ such that $\lambda_0=\tilde \lambda_1-df$.
It follows directly that $\rho^* df=-df$.
If $\rho^*f\neq-f$, then we replace $f$ by $\frac{1}{2}(f-\rho^*f)$ .

Find a $\rho$-invariant cutoff function $g$ such that $g\equiv 1$ on $W_0$, and such that $g\equiv 0$ on the complement of a neighborhood of $W_0$.
Then $\lambda_1=\tilde \lambda_1-d(gf)$ has the desired properties.
\end{proof}

\begin{lemma}
\label{lemma:holomorphic_disk}
Let $(M,\omega,\mathcal D=(\Sigma,A,S)\,)$ be a decorated symplectic manifold, and assume that $(W_0,\lambda_0,\rho|_{W_0})$ is a real Liouville domain that embeds into the interior of $M-\nu(\Sigma)$ for some $\rho$-invariant neighborhood $\nu(\Sigma)$ of $\Sigma$.
Suppose in addition that $b_1(W_0)=0$.
Let $J$ be an almost complex structure on $M$ that is compatible with $\omega$, and SFT-like near $\partial W_0$.

Assume that $u:\mathbb C \mathbb P^1\to M$ is a $J$-holomorphic sphere through a point $p\in W_0$ such that $[u]\circ [\Sigma]=1$.
Then the component $C$ of $u^{-1}(W_0)$ containing $z_0$ with $u(z_0)=p$ satisfies the following:
\begin{itemize}
\item $C$ is diffeomorphic to a disk.
\item $\int_C u|_{C}^*\omega=\int_C u|_{C}^*d\lambda_0 \leq 1$.
In particular, the SFT energy of $u|_C$ is bounded from above by $1$.
\end{itemize}   
\end{lemma}

\begin{proof}
After possibly shifting the boundary $\partial W_0$ a little, we can assume that $u^{-1}(\partial W_0)$ consists of finitely many circles.
Let $C$ denote the component of $u^{-1}(W_0)$ containing $z_0$.
We claim that $C$ has only one boundary component.
To see why, note that if $C$ has more than one boundary component, then there is a connected component of $\tilde C:=\mathbb C \mathbb P^1-int(C)$ with the properties
\begin{itemize}
\item $\tilde C$ shares a boundary component with $C$.
\item $u(\tilde C)$ does not intersect $\nu(\Sigma)$, and is contained in $M-int(W_0)$.
\end{itemize}
To see that the latter condition can be imposed, we observe that $u$ intersects $\Sigma$ only once, and we also use that $\mathbb C \mathbb P^1$ has genus $0$.

Now apply the previous lemma to see that $M-\nu(\Sigma)$ carries the structure of a real Liouville domain $(W_1,\lambda_1)$ with real Liouville subdomain $(W,\lambda)$. 
This allows us to compute the energy of $\tilde C$ via Stokes' theorem,
$$
E(u|_{\tilde C})=\int_{\tilde C} u_{\tilde C}^*\omega
=\int_{\tilde C} d u_{\tilde C}^* \lambda_1
=\int_{\partial \tilde C} u_{\tilde C}^* \lambda_1<0.
$$
The last inequality holds, because the orientation induced by the outward pointing normal is minus the one induced by the Reeb vector field; one can see this by using that $J$ is SFT-like near $\partial W_0$.
Since the energy of the holomorphic curve $u|_{\tilde C}$ is positive, this is a contradiction, so we conclude that $C$ has one boundary component.
It follows directly that $C$ is diffeomorphic to a disk.
The claimed energy estimate is now also clear since $\int_{\mathbb C \mathbb P^1} u^*\omega=1$.
\end{proof}

We will now apply a stretching argument to obtain an invariant finite energy plane.
This is illustrated in Figure~\ref{fig:holomorphic_spheres}.
\begin{figure}[htp]
\def\svgwidth{0.25\textwidth}%
\begingroup\endlinechar=-1
\resizebox{0.25\textwidth}{!}{%
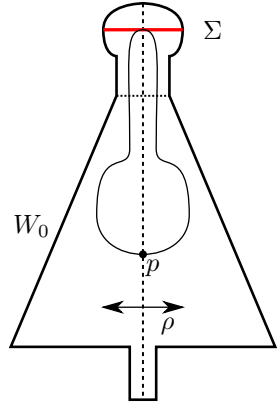%
}\endgroup
\caption{Hanging up Christmas balls (holomorphic spheres) in a Christmas tree}
\label{fig:holomorphic_spheres}
\end{figure}
Let $X$ denote the Liouville vector field on $M-\Sigma$.
Take a point $p\in W_0$, and for $\tau\in \R_{\geq 0}$ define $p_\tau$ by following the Liouville flow backwards, $p_\tau=Fl^X_{-\tau}(p)$.
Define the stretched Liouville domain $W_0^\tau$ by
$$
W_0^\tau:=(W_0,\omega=d\lambda_0) \cup_\partial ([0,\tau]\times \partial W_0,d(e^t \lambda_0|_{\partial W_0}) \,)
.
$$
Choose a compatible complex structure $J_\tau$ on $W_0^\tau$ that is SFT-like on $[0,\tau]\times \partial W_0$.
We choose this sequence $J_\tau$ such that it is a constant sequence of complex structures when restricted to $W_0$.
Since the map $x\mapsto Fl^X_\tau(x)$ provides a symplectic deformation from $W_0$ to $W_0^\tau$, we can pull back $J_\tau$ to a complex structure on $W_0$ that is SFT-like near the boundary. Extend this $J_\tau$ to a compatible complex structure $\tilde J_\tau$ for $(M,\omega)$.

With Lemma~\ref{lemma:holomorphic_disk} applied to an invariant holomorphic sphere obtained with Theorem~\ref{realmain}, we find a $\tilde J_\tau$-holomorphic disk
$$
\tilde u_\tau:\tilde C_\tau\subset \C \P^1 \longrightarrow W_0
$$
going through $p_\tau$, and with boundary on $\partial W_0$.
We now stretch the Liouville domain $W_0$ to a Liouville domain $W_0^\tau$ using the above deformation.
This deformation also gives us a $J_\tau$-holomorphic curve
$$
u_\tau: C_\tau \longrightarrow W_0^\tau
$$
going through $p$.
As the Hofer energy of $\tilde u_\tau$ is bounded by $1$, so is the Hofer energy of $u_\tau$.

Denote the norm induced by $\omega_\tau(\cdot,J_\tau \cdot)$ by $\Vert \cdot \Vert_{\tau}$.
By rescaling the domain we can ensure that $\max_{z\in C_\tau} \Vert du_\tau \Vert_\tau=1$; we need to rescale the disk $C_\tau$ for this, but we will continue to write $C_\tau$ for this rescaled disk.
Since $p$ lies in $W_0$ and the boundary of the disk, $u_\tau(\partial C_\tau)$, lies on $\{ \tau \} \times \partial W_0$, we see directly that radius for the disk $C_\tau$ has to be at least $\tau$ by a very crude estimate using $\max_{z\in C_\tau} \Vert du_\tau \Vert_\tau=1$.
Taking the limit $\tau\to \infty$, we find a convergent subspace, and obtain a map
$$
u_\infty: \C \to W_0^\infty,
$$
where $W_0^\infty$ is the completion of $W_0$.
As the Hofer energy of $u_\infty$ is bounded from above by $1$, we conclude that $u_\infty$ is the desired finite energy plane through $p\in W_0$.

This stretching construction also implies the well-known corollary, see also \cite{lu}.
\begin{cor}
Let $(W,\lambda)$ be a Liouville domain admitting an embedding into a decorated symplectic manifold $(M,\omega,\mathcal{D})$.
Suppose that $b_1(W)=0$. Then $W$ is uniruled.
Furthermore, there exists a periodic orbit of period less than or equal to 1 for the Reeb flow on
$\partial W$.
\end{cor}

\noindent
{\bf Acknowledgments.}
We thank Jungsoo Kang for helpful comments.
The second author was partially supported by the NRF Grant 2012-011755 funded by the Korean government. Both authors also hold joint appointments in the Research Institute of Mathematics, Seoul National University.

\end{document}